\documentclass[11pt,reqno]{amsart}

\usepackage{amssymb}
\usepackage{amsmath}
\usepackage{amsthm}
\usepackage{mathrsfs}
\usepackage{mathpazo}
\usepackage{color}
\usepackage[xcolor=pst]{pstricks}
\usepackage{tikz}
\usepackage{cite}

\newtheorem{thm}{Theorem}
\newtheorem{pro}[thm]{Proposition}
\newtheorem{lem}[thm]{Lemma}

\theoremstyle{remark}
\newtheorem{rmk}[thm]{Remark}

\newrgbcolor{darkred}{0.80 0 0}
\newrgbcolor{darkblue}{0.05 0.3 0.7}

\def\J{\mathscr{J}}
\def\D{\mathscr{D}}
\def\R{\mathscr{R}}
\def\L{\mathscr{L}}
\def\H{\mathscr{H}}

\def\blue#1{\mathbf{b}({#1})}
\def\red#1{\mathbf{r}({#1})}
\def\cc#1{{\mathbf{c}({#1})}}

\newcommand{\fillbox}[2]{\draw[fill=gray!30](#1,#2)--(#1+1,#2)--(#1+1,#2+1)--(#1,#2+1)--(#1,#2);}
\newcommand{\rank}{\operatorname{rank}}
\newcommand{\idrank}{\operatorname{idrank}}

\begin{document}

\begin{flushleft}

{\Large{\color{darkblue}{\bf \emph{Igor Dolinka\footnote{Department of Mathematics and Informatics, University of Novi Sad, Trg Dositeja Obradovi\'ca 4, 21101 Novi Sad, Serbia, {\tt dockie\,@\,dmi.uns.ac.rs}.} \& James East\footnote{Centre for Research in Mathematics; School of Computing, Engineering and Mathematics, Western Sydney University, Locked Bag 1797, Penrith NSW 2751, Australia, {\tt J.East\,@\,WesternSydney.edu.au}.}}}\vspace{2mm}\hrule}}

\end{flushleft}

\vspace{2mm}

\begin{flushright}

{\Large {\bf \emph{\color{darkred}{The idempotent generated subsemigroup \\ of the Kauffman monoid}}}}

\end{flushright}

\bigskip\bigskip\bigskip


\begin{quote}
{\bf Abstract.}  
We characterise the elements of the (maximum) idempotent generated subsemigroup of the Kauffman monoid in terms of combinatorial data associated to certain normal forms.  We also calculate the smallest size of a generating set and idempotent generating set.


\end{quote}

\section{Introduction}

Let $n\geq 2$ and $c\in\mathbb{C}$ (more generally, instead of complex numbers $\mathbb{C}$ one can take an arbitrary
commutative ring $R$). The \emph{Temperley-Lieb algebra} $TL_n(c)$, introduced in \cite{TL},  is the unitary associative algebra given by the 
presentation consisting of generators $h_1,\dots,h_{n-1}$ and defining relations
\begin{align*}
h_ih_j &= h_jh_i && \text{whenever }|i-j|\geq 2,\\
h_ih_jh_i &= h_i && \text{whenever }|i-j|=1,\\
h_i^2 = ch_i&=h_ic &&\text{for all }1\leq i<n.
\end{align*}
When $c=1$, we obtain a special case, the so-called \emph{Jones algebra} \cite{Jo}, and its basis forms
a monoid called the \emph{Jones monoid} $J_n$ \cite{EG,LFG}. Elements of the Jones monoid form the basis of general
Temperley-Lieb algebras as well, with the exception that within $TL_n(c)$ they need not form a monoid anymore, 
as witnessed by the third relation above. (Indeed, the Temperley-Lieb algebra is the \emph{twisted semigroup algebra}
of the Jones monoid; see \cite{Wi}.) However, it is possible to `extract' a monoid from the Temperley-Lieb algebra by considering
the above presentation as a \emph{monoid} presentation -- which is indeed possible, as it contains no mention of 
the addition operation -- including $c$ as a separate monoid generator. (Henceforth, generation is always within the variety of monoids unless otherwise specified.) In this way, we obtain the \emph{Kauffman 
monoid} $K_n$, which (upon interpretation of the symbol $c$ as a scalar multiple of $1$) spans $TL_n(c)$
but is not a basis (e.g. due to $c$ and $1$ not being independent).

The name was coined in the paper 
\cite{BDP} in honour of Louis H. Kauffman who was the first to realise the connection between planar Brauer diagrams
and the Temperley-Lieb algebra \cite{Ka}, although the first full, self-contained proof of isomorphism between $K_n$ and 
the monoid consisting of pairs $(c^k,\alpha)$ where $k\geq 0$ is an integer and $\alpha$ is a planar Brauer diagram
is given in \cite{BDP}. The operation in the latter monoid -- naturally, also called the Kauffman monoid -- is defined 
by $(c^k,\alpha)(c^\ell,\beta)=(c^{k+\ell+\tau(\alpha,\beta)},\alpha\beta)$, where $\tau(\alpha,\beta)$ is the number
of inner circles formed in the course of computing the product $\alpha\beta$ in the Brauer monoid by stacking
$\alpha$ on top of $\beta$. For $c=1$ we get that the Jones monoid is isomorphic just to the planar submonoid of the 
Brauer monoid. In such a diagrammatic representation $c$ is just the pair $(c,1)$, while $h_i$ is interpreted as
$(c^0,\delta_i)$, where $\delta_i$ is the \emph{hook} (or \emph{diapsis}): its connected components are $\{i,i+1\}$,
$\{i',(i+1)'\}$ and $\{j,j'\}$ for all $j\not\in\{i,i+1\}$.  Any equation in the current paper may be verified using these diagrams, but we find the approach via words and presentations to be more convenient.  At only one point (in the proof of Lemma~\ref{lem:J}) will we rely on a (very simple) diagrammatic calculation.  For more on diagrams, see for example \cite{BDP,LFG,DEtwisted}.

Beyond the above-mentioned article \cite{BDP}, a number of previous studies of the Kauffman monoid have been carried out.  
Gr\"obner-Shirshov bases are discussed in \cite{BL}.  Green's relations and the ideal structure of $K_n$ (and associated quotients) are described in \cite{LFG}.  In \cite{ACHLV}, it is shown that $K_n$, with $n\geq3$, has no finite basis for its identities (considered either as a semigroup or as an involution semigroup).  The idempotents of $K_n$ (and other planar diagram monoids) are classified and enumerated in \cite{DEEFHHL2}.
In the current work, we describe the idempotent generated subsemigroup of $K_n$ (Theorem \ref{main1}).  We also calculate the rank (smallest size of a generating set) and idempotent rank (smallest size of an idempotent generating set) of this subsemigroup (Theorem \ref{main2}).  We note that these tasks have been carried out for a number of related diagram monoids, such as the (twisted) Brauer, Jones, Motzkin and partition monoids; see for example \cite{DEtwisted,DEG,Epnsn,EG,MM}.  The original studies of idempotent generated subsemigroups in  full transformation semigroups may be found in \cite{Howie1966,Howie1978}; see also \cite{Erdos1967}.  However, in contrast to many of these examples, the rank and idempotent rank are not equal (apart from small cases) when it comes to the idempotent generated subsemigroup of $K_n$.

%
If $n\leq2$, then $K_n$ has a unique idempotent (the identity element), so we assume $n\geq3$ throughout.

\section{Preliminaries}\label{sect:prelim}

We now describe the \emph{Jones normal forms} given in \cite{BDP}.
These are given in terms of \emph{blocks}, which are defined to be words of the form
$$
h[j,i] = h_jh_{j-1}\dots h_{i+1}h_i
$$
for any $1\leq i\leq j<n$. Also, with the same assumptions on $i,j$ we define an \emph{inverse block} to be a
word of the form $h[i,j]=h_ih_{i+1}\dots h_j$. Note that $h[i,i]=h_i$, which exhausts all blocks that are also
inverse blocks. A block $h[j,i]$ will be called \emph{white} if $i$ and $j$ are of different parity. If both
$i,j$ are odd, then the block $h[j,i]$ is called \emph{blue}, otherwise (if both $i,j$ are even) it is called
\emph{red}. Analogous naming conventions hold for inverse blocks, too.

An element $w\in K_n$ (represented as a word over $\{c,h_1,\dots,h_{n-1}\}$) is said to be in \emph{Jones
normal form} \cite{BDP} (J.n.f.\ for short) if it has the form
$$
c^\ell h[b_1,a_1]\dots h[b_k,a_k]
$$
for some $k,\ell\geq 0$ and increasing sequences $a_1<\cdots<a_k$ and $b_1<\cdots<b_k$. The first principal result
of Borisavljevi\'c, \emph{et.\ al.} \cite[Lemma 1]{BDP} is that every element of $K_n$ is equivalent to a \emph{unique} word 
in J.n.f.

Here we give a digest of their argument, in fact a part of it that is relevant to this note. The first step is
to change the generating set and provide a different presentation for $K_n$. This new generating set will consist
of $c$ and all the blocks $h[j,i]$ (this set trivially generates $K_n$ as it contains all singleton blocks
$h[i,i]=h_i$). Then, a standard argument is provided to show that this new, enlarged set of generators, along with
relations
\begin{align}
h[j,i]h[l,k] &= h[l,k]h[j,i] && \text{whenever }i\geq l+2,\\
h[j,i]h[l,k] &= h[j,k]&& \text{whenever }j\geq k\text{ and }|i-l|=1,\\
h[j,i]h[i,k] &= ch[j,k] &&\text{for all }1\leq k\leq i\leq j\leq n,\\
h[j,i]c &= ch[j,i] &&\text{for all }1\leq i\leq j<n,
\intertext{also define $K_n$. Furthermore, three additional groups of relations were deduced as consequences for $i+2\leq l$:}
h[j,i]h[l,k] &= h[l-2,k]h[j,i+2] && \text{if }j\geq l\text{ and }i\geq k,\\
h[j,i]h[l,k] &= h[j,k]h[l,i+2]&& \text{if }j<l\text{ and }i\geq k,\\
h[j,i]h[l,k] &= h[l-2,i]h[j,k] &&\text{if }j\geq l\text{ and }i<k.
\end{align}
Here is the gist of the argument from \cite{BDP} (clearly contained in the proof of their Lemma 1), which directly shows 
the statement about J.n.f.'s.

\begin{lem}\label{rewr}
Let $\Sigma$ be the rewriting system on words over the alphabet consisting of $c$ and all blocks, obtained by
orienting all the defining relations (1)--(7) from left to right. Then $\Sigma$ is confluent and Noetherian (and thus 
every word has a unique normal form). The normal forms of $\Sigma$ are precisely the J.n.f.'s. \hfill$\Box$
\end{lem}

If $u,v$ are words in the blocks, we write $u\to v$ if $u=u_1xu_2$ and $v=u_1yu_2$ for words $u_1,u_2,x,y$, and where $x$ and $y$ occur on the left and right hand sides of one of equations (1)--(7), respectively.  We write $\to^*$ for the transitive closure of $\to$.  The previous lemma says not only that for any word $u$, $u\to^* v$ for some J.n.f.~$v$.  It says that \emph{any} sequence $u\to u_1\to u_2\to\cdots$ will eventually terminate in a J.n.f., and that this J.n.f.~will be unique.

While working within $\Sigma$, we will freely use inverse blocks $h[i,j]$, $i\leq j$ where the latter is now simply a 
short-hand for the word $h[i,i]\dots h[j,j]$. Also, where appropriate, we will freely use the connection between new and
old generators, because the old generators are (up to renaming) a subset of the new ones, and the connection can be
deduced within $\Sigma$. 

\section{The idempotent generated subsemigroup}\label{sect:IGS}

The set of all idempotent elements of $K_n$ (written via blocks or otherwise)
we write as~$E_n$.  The goal of this section is to describe the elements of $\langle E_n\rangle$, the idempotent generated subsemigroup of $K_n$; see Theorem \ref{main1}.  We do this in three main steps; see Propositions \ref{direct} and \ref{converse} and Lemma \ref{final}.

By $E'_n$ we denote the subset of $E_n$ consisting of all blocks and inverse blocks of length~2, namely
$h[i+1,i]$ and $h[i,i+1]=h[i,i]h[i+1,i+1]$ (by the length of a(n inverse) block $h[j,i]$ we mean $|i-j|+1$). Of course, these 
are trivially checked to be idempotents, as, for example $h[i+1,i]^2=h_{i+1}h_ih_{i+1}h_i=h_{i+1}h_i$. This easily generalises
to the following statement, which we record for completeness.

\begin{lem}\label{white}
A(n inverse) white block is a product of elements of $E'_n$.
\end{lem}

\begin{proof}
If $j\geq i$ are of different parity, then
$$
h[j,i]=h_jh_{j-1}\dots h_{i+1}h_i = h[j,j-1]\dots h[i+1,i].
$$
The argument for inverse blocks is analogous.
\end{proof}

\begin{lem}\label{parity}
If $k,l$ are of different parity then $h_kh_l$ is a product of elements of $E'_n$. 
\end{lem}

\begin{proof}
Assume that $k>l$. If $k=l+1$, then the result is trivial, while if $k\geq l+2$, then 
\begin{align*}
h_kh_l &= (h_kh_{k-1}\dots h_{l+2}h_{l+1}h_{l+2}\dots h_k)h_l
= (h_k\dots h_{l+1}h_l)(h_{l+2}\dots h_k) = h[k,l]h[l+2,k],
\end{align*}
a product of a white block and a white inverse block; hence, the lemma follows from Lemma \ref{white}.
The argument is analogous if $k<l$.
\end{proof}

We are now in position to show the first of the three main steps leading to the characterisation of 
$\langle E_n\rangle$. To this end, for a word $w$ over the alphabet consisting of $c$ and the blocks,
let $\blue{w}$ be the number of blue blocks occurring in $w$; similarly, let $\red{w}$ count the number 
red blocks in $w$, while $\cc{w}$ is simply $|w|_c$, the number of occurrences of $c$ in $w$. We define 
the \emph{characteristic number} of $w$ as
$$
\chi(w) = \cc{w} - |\blue{w}-\red{w}|.
$$

\begin{pro}\label{direct}
Let $w$ be a J.n.f. that is equal (in $K_n$) to a product of idempotents from $E'_n$. Then $\chi(w)$ is
non-negative and even.
\end{pro}

\begin{proof}
If $w$ is a J.n.f. equal to a product of elements from $E'_n$ then there exists a word $w'$ consisting of factors
of the form $h[i+1,i]$ and $h[i,i+1]=h[i,i]h[i+1,i+1]$ such that $w=w'$ holds in $K_n$. Note that these factors
are either white, or blue-red, or red-blue; in any case, their characteristic numbers are $0$. Therefore, 
$\chi(w')=0$. By Lemma \ref{rewr}, $w'\to^\ast w$ holds in $\Sigma$, so there is a finite sequence of rewriting
rules stemming from (1)--(7) which transform $w'$ into $w$. So, our proposition will be proved once we show
that an application of any of these rules in the course of a single step $u\to v$ neither decreases, nor changes
the parity of the characteristic number.

In fact, we claim that $\chi(v)-\chi(u)\in\{0,2\}$, which can be verified by direct inspection of the rules.
It is easy to see that by applying any of the rules (1)-(2) and (4)-(7) we have $\cc{u}=\cc{v}$ and one of the
following happens:
\begin{itemize}
\item[(i)] one or more white blocks are created from a pair of blue and red blocks, or
\item[(ii)] a pair of blue and red blocks is created from a pair of white blocks, or
\item[(iii)] the number of blue and red blocks involved is unchanged.
\end{itemize}
Hence, in all these cases we have $|\blue{u}-\red{u}|=|\blue{v}-\red{v}|$ and so $\chi(u)=\chi(v)$. So, the only
`interesting' rule is (3). Here, one of the following three things can happen:
\begin{itemize}
\item[(i)] the rule takes two white blocks and turns them into one $c$ and one block that is either blue or red, or
\item[(ii)] the rule takes either two blue or two red blocks and turns them into one $c$ and one block of the same 
colour as the initial two, or
\item[(iii)]  the rule takes a white block and a non-white block and turns them into a $c$ and a white block.
\end{itemize}
Any of the above three operations either leaves the characteristic number of a word unchanged, or increases it by $2$.
This completes the proof of the proposition.
\end{proof}

Our next aim is to prove the converse of Proposition \ref{direct}: if $w$ is a J.n.f. such that ${\chi(w)\geq 0}$ is even, 
then $w$ is equivalent to a product of elements of $E'_n$. For this we need three additional lemmas, the third one 
being a folklore exercise in combinatorics on words.

\begin{lem}\label{cb-cr}
Let $h[j,i]$ be a block that is not white (so that $i,j$ are of the same parity). Then $ch[j,i]$ is a product
of elements of $E'_n$.
\end{lem}

\begin{proof}
If $i=j>1$ we have $h[i,i-1]h[i-1,i]=h_ih_{i-1}^2h_i=ch_ih_{i-1}h_i=ch_i$ (if $i=1$ we may use $h_{i+1}$ instead
of $h_{i-1}$). Otherwise, we have
$$
h[j,j-1]h[j-1,j]h[j-1,i]=ch_jh[j-1,i]=ch[j,i],
$$
so the lemma follows from Lemma \ref{white}, bearing in mind that $h[j-1,i]$ is white.
\end{proof}

\begin{lem}\label{c+2}
If the word $w$ is equivalent to a product of elements from $E'_n$ so is $c^2w$.
\end{lem}

\begin{proof}
Without loss of generality, assume that $w=h[i+1,i]w'$ holds in $K_n$ for some word $w'$ over $E'_n$. Then
$$
c^2w = c^2h[i+1,i]w' = h[i+1,i]h[i,i+1]h[i+1,i]w',
$$
and we are done.
\end{proof}

For the next lemma, if $v$ is a word over $\{0,1\}$, we write $|v|$, $|v|_0$ and $|v|_1$ for the length of $v$, the number of $0$'s in $v$ and the number of $1$'s in $v$, respectively.

\begin{lem}\label{comb}
A word $v$ over $\{0,1\}$ is called \emph{balanced} if $|v|_0=|v|_1$. Let $u\in\{0,1\}^\ast$ such that 
$|u|_0-|u|_1=k\geq 0$. Then $u$ can be factorised into a product of balanced words 
and words containing only $0$'s such that the total length of the latter is equal to~$k$.
\end{lem}

\begin{proof}
For a word $v$ over $\{0,1\}$, write $k(v)=|v|_0-|v|_1$.  We prove the lemma by induction on $|u|+k(u)$.  If $k(u)=0$, then the result is trivial; this includes the base case of the induction, in which $|u|+k(u)=0$.  Now assume that $k(u)\geq1$ (so also $|u|\geq1$).  Write $u=x_1\cdots x_m$, where each $x_i\in\{0,1\}$.  If $x_1=0$, then $k(x_2\cdots x_m)=k(u)-1$, and an induction hypothesis completes the proof in this case.  If $x_1=1$, then, since $k(u)\geq0$, there exists $2\leq r\leq m$ such that $k(x_1\cdots x_r)=0$ (i.e., $x_1\cdots x_r$ is balanced).  But then $u=(x_1\cdots x_r)(x_{r+1}\cdots x_m)$, with $k(x_{r+1}\cdots x_m)=k(u)$, and we are again done after applying an induction hypothesis.
\end{proof}

%
%

\begin{pro}\label{converse}
Let $w$ be a J.n.f.\ such that $\chi(w)\geq 0$ is even. Then $w$ is equal to a product of elements from $E'_n$.
\end{pro}

\begin{proof}
We begin by several reductions of the statement to its special cases. First of all, we can assume without loss
of generality that $\chi(w)=0$. Indeed, write $w=c^{\cc{w}}w'$, where $w'$ is the part of $w$ containing no
occurrences of $c$. Then
$$
w=c^{\chi(w)}c^{|\blue{w}-\red{w}|}w',
$$
so if were able to prove that $c^{|\blue{w}-\red{w}|}w'$ is a product of elements of $E'_n$, the same would be
true for $w$ by repeated applications of Lemma \ref{c+2} (since $\chi(w)$ is even).

Furthermore, call a J.n.f.\ \emph{tightly balanced} if it contains no occurrences of $c$, has the same number of 
blue and red blocks, and cannot be factorised into shorter J.n.f.'s with the previous two properties (if the J.n.f.\ is not simply a single white block, this
necessarily implies that neither its first nor its last blocks can be white, in fact, exactly one of them is blue 
and the other is red). We claim that it suffices to prove the statement of the proposition for tightly balanced 
J.n.f.'s only. Indeed, let $w$ be an arbitrary J.n.f.\ such that $\chi(w)=0$. Without loss of generality, assume 
that $\blue{w}\geq\red{w}$ (otherwise just switch the roles of blue and red). 
Form a binary sequence by inspecting $w$ from left to right, ignore every $c$ and every white block,
writing down a $0$ for each blue block and $1$ for each red block. We end up with a word $u$ where $|u|_0-|u|_1=
\blue{w}-\red{w}=\cc{w}$. By Lemma \ref{comb}, there is a factorisation of $u$ such that each factor is either
a balanced word, or a sequence of $0$'s. Furthermore, we may assume that this factorisation is maximal in the
sense that none of the balanced words involved can be factorised further into balanced factors (such factors
must have different first and last letters). Then, to each factor $u'$ of $u$ that is a balanced word, there 
naturally corresponds a factor of $w$ that is a tightly balanced J.n.f. (by starting with the non-white block 
inducing the  first letter of $u'$ and concluding with the also non-white block inducing the last letter of $u'$; 
note that this may involve a number of white blocks in between). What is left outside these tightly balanced 
factors of $w$ is $c^{\cc{w}}$, $\cc{w}$ stand-alone blue blocks (corresponding to stand-alone $0$'s in $u$) and 
an unspecified number of white blocks. By commuting the $c$'s next to these stand-alone blue blocks, we conclude 
that $w$ can be written as a product of two types of factors:
\begin{itemize}
\item tightly balanced J.n.f.'s (including white blocks),
\item blue blocks multiplied by $c$.
\end{itemize}
Thus, if we were able to prove the proposition for tightly balanced blocks, the general case would follow
immediately by Lemma \ref{cb-cr}.

So, assume that $w=h[b_1,a_1] \dots h[b_r,a_r]$ is a tightly balanced J.n.f.; here $r$ is called the \emph{weight}
of $w$. We proceed by induction on $r$. If $r=1$, then $w$ is just a white block, whence we are done by Lemma
\ref{white}. Hence, assume that $r\geq 2$ and that all tightly balanced J.n.f.'s of weight $<r$ are indeed products
of elements of $E'_n$. There will be no loss of generality in assuming that $h[b_1,a_1]$ is blue, so that $a_1,b_1$
are odd. By the tightly balanced condition, $h[b_r,a_r]$ is then red.

We call a J.n.f.\ $h[d_1,c_1]h[d_2,c_2]\dots h[d_s,c_s]$ a \emph{stairway} if $c_{i+1}-c_i=1$ for all $1\leq i<s$. 
Let $q$ be the length of the maximal prefix of $w$ that is a stairway; so, $a_i=a_1+i-1$ for $1\leq i\leq q$, 
but $a_{q+1}\geq a_q+2$ (or, alternatively, there's no such $a_{q+1}$ at all if $r=q$). Then, the principal idea is
to `shave off' the bottoms of the blocks belonging to this maximal initial stairway of $w$ and `float' them to
the right; more precisely, we have:
\begin{align*}
w &= h[b_1,a_1]h[b_2,a_2] \dots h[b_q,a_q]h[b_{q+1},a_{q+1}] \dots h[b_r,a_r]\\ 
  &= (H[b_1,a_1+1]h_{a_1})(H[b_2,a_2+1]h_{a_2}) \dots (H[b_q,a_q+1]h_{a_q}) h[b_{q+1},a_{q+1}] \dots h[b_r,a_r] \\
	&= \Big(H[b_1,a_1+1]H[b_2,a_2+1] \dots H[b_q,a_q+1]h[b_{q+1},a_{q+1}] \dots h[b_r,a_r]\Big)h_{a_1}\dots h_{a_q}\\
	&= \Big(H[b_1,a_1+1]H[b_2,a_2+1] \dots H[b_q,a_q+1]h[b_{q+1},a_{q+1}] \dots h[b_r,a_r]\Big)h[a_1,a_q],
\end{align*}
where $H[b_s,a_s+1]$ is $h[b_s,a_s+1]$ if $b_s>a_s$ and an empty word otherwise.  Notice here that $h[a_1,a_q]$ is an inverse block of length $q$, and the expression in the parenthesis in the last displayed line is a J.n.f.\ of weight $\leq r$.


Now we consider two cases depending on the parity of $q$, noting that this is the same as the parity of $a_q$. First, let $q$ be odd. In that case we cannot have $q=r$
(because $a_r$ is even), so we can transform $w$ further into
\begin{align*}
w &= H[b_1,a_1+1]\Big(H[b_2,a_2+1] \dots H[b_q,a_q+1]h[b_{q+1},a_{q+1}] \dots h[b_{r-1},a_{r-1}]\Big)\times\\
  & \quad\quad\quad\quad\quad\quad\times  H[b_r,a_r+1] (h_{a_r}h_{a_1}) H[a_1+1,a_q],
\end{align*}
with a similar convention about the use of $H$ in inverse blocks. Here, all three capital $H$'s outside the 
parentheses are white blocks or inverse blocks or empty, so they are products of elements from $E'_n$, as is
$h_{a_r}h_{a_1}$ by Lemma \ref{parity}. Hence, it suffices to show that the word within the parentheses is a product of elements of $E_n'$.  To do this, we will focus on how the colours of the blocks within the
parenthesis changed. By replacing $h[b_s,a_s]$ ($2\leq s\leq q$) by $H[b_s,a_s+1]$, any blue or red block
either turns white or vanishes altogether. In turn, a white block is turned blue if $s$ is even and red if
$s$ is odd. Also, notice that $h[b_s,a_s]$ can be blue only if $s$ is odd, while it can be red only if $s$ is
even. In other words, for even values of $s$, white blocks turn blue and red blocks turn white (or they disappear),
while for odd values of $s$ white blocks turn red and blue blocks turn white (or they vanish). 
So, if there were $m$ blue and $p$ red blocks among $h[b_s,a_s]$, $2\leq s\leq q$, then after the 
`shaving off' procedure we have $(q-1)/2-p$ blue blocks and $(q-1)/2-m$ red blocks among $H[b_s,a_s]$, 
$2\leq s\leq q$. However, note that the \emph{difference} between the number of blue and red blocks has not
changed at all by transforming $h[b_2,a_2] \dots h[b_q,a_q]$ into $H[b_2,a_2+1] \dots H[b_q,a_q+1]$; in both cases
it is $|m-p|$. This suffices to conclude that the J.n.f.
$$
H[b_2,a_2+1] \dots H[b_q,a_q+1]h[b_{q+1},a_{q+1}] \dots h[b_{r-1},a_{r-1}]
$$
has an equal number of blue and red blocks (because such was $$h[b_2,a_2] \dots h[b_{r-1},a_{r-1}],$$ which is 
just the original J.n.f.\ $w$ stripped of its outermost blocks), and hence, by Lemma~\ref{comb} and the previously 
presented reduction to the case of tightly balanced J.n.f.'s, it is a product of tightly balanced J.n.f.'s of 
weight $<r$ (since its total weight is $\leq r-2$). By induction hypothesis, it is a product of elements of $E'_n$.

Finally, suppose $q$ is even.  Recall  that $w=H[b_1,a_1+1]w'h[a_1,a_q]$, where
$$w' = H[b_2,a_2+1] \dots H[b_q,a_q+1]h[b_{q+1},a_{q+1}] \dots h[b_r,a_r]$$
is a J.n.f.\ of weight $<r$.
This time, $h[a_1,a_q]$ is a white inverse block, and so a product of elements of $E_n'$, by Lemma \ref{white}.  A counting argument analogous to the previous case shows that $H[b_1,a_1+1]w'$ has the same number of blue and red blocks.  But $H[b_1,a_1+1]$ is still either empty or a white block, so it follows that $w'$ has the same number of blue and red blocks, and the proof concludes as in the previous case.
\end{proof}

%

Everything is in place to lay out the third ingredient, showing that $\langle E_n\rangle=\langle E'_n\rangle$. 
For this, it suffices to show that every idempotent of $K_n$ is a product of elements from $E'_n$, by arguing
that it falls under the scope of the previous proposition.

\begin{lem}\label{final}
Let $w$ be a J.n.f.\ representing an element of $E_n$. Then $\cc{w}=0$ and $\blue{w}=\red{w}$.
\end{lem}

\begin{proof}
The conclusion $\cc{w}=0$ is immediate. A direct consequence of this is that $\chi(ww)=2\chi(w)$. However, in
$\Sigma$ we have $ww\to^\ast w$, and thus, by the argument from the proof of Proposition \ref{direct}, we get
$$
2\chi(w)=\chi(ww)\leq \chi(w).
$$
This is possible only if $\chi(w)=|\blue{w}-\red{w}|=0$, so the lemma follows.
\end{proof}

Summing up, we have proved the following result.

\begin{thm}\label{main1}
Assume $w\in K_n$ is represented in its Jones normal form. Then $w\in\langle E_n\rangle$ (the idempotent generated
subsemigroup of $K_n$) if and only if $\chi(w)$ is non-negative and even. \hfill$\Box$
\end{thm}

\section{Rank and idempotent rank}\label{sect:RIR}

Recall that the \emph{rank}, $\rank(M)$, of a monoid $M$ is the least cardinality of a (monoid) generating set for $M$.  If $M$ is idempotent generated, the \emph{idempotent rank}, $\idrank(M)$, is defined analogously in terms of generating sets consisting of idempotents.  In this final section, we calculate the rank and idempotent rank of $\langle E_n\rangle$.  Before we do this, we first need to recall some ideas from semigroup theory.  For more details, the reader may consult Howie's monograph \cite{Howie}.

With this in mind, let $S$ be a semigroup, and let $S^1$ be the result of adjoining an identity element to $S$ if $S$ was not already a monoid.  Recall that \emph{Green's relations} $\R,\L,\J,\H,\D$ are defined on $S$ by
\begin{gather*}
x\,\R\,y \Leftrightarrow xS^1=yS^1,\quad 
x\,\L\,y \Leftrightarrow S^1x=S^1y,\quad  
x\,\J\,y \Leftrightarrow S^1xS^1=S^1yS^1,\\
\H=\R\cap\L,\quad  \D=\R\circ\L=\L\circ\R.
\end{gather*}
If $x\in S$, we write $J_x$ for the $\J$-class of $S$ containing $x$.  The $\J$-classes of $S$ are partially ordered by $J_x\leq J_y \Leftrightarrow x\in S^1yS^1$.  If $J$ is a $\J$-class of $S$, then the \emph{principal factor} of $J$ is the semigroup $J^\star$ defined on the set $J\cup\{0\}$, where $0$ is a new symbol not belonging to $J$, and with product $\star$ defined by
\[
x\star y = \begin{cases}
xy &\text{if $x,y,xy\in J$}\\
0 &\text{otherwise.}
\end{cases}
\]
As noted in \cite{Gr}, if $S$ is generated as a semigroup by a subset $X\subseteq S$, then clearly $X$ contains a generating set for the principal factor of any maximal $\J$-class.

Green's relations on $K_n$ are characterised (in terms of the diagrammatic representation) in \cite{LFG}.  We will not need to recall these characterisations in their entirety.  But of importance is that the $\D$ and $\J$ relations coincide, that the $\H$ relation is the equality relation, that $\{1\}$ is the unique maximal $\J$-class, that the set $D=\{h[i,j]:1\leq i,j<n\}$ consisting of all blocks and inverse blocks is a $\D$-class, and that
\[
h[i,j]\,\R\, h[k,l] \Leftrightarrow i=k \quad\text{and}\quad
h[i,j]\,\L\, h[k,l] \Leftrightarrow j=l .
\]
Note that, by Theorem \ref{main1}, 
\[
D\cap\langle E_n\rangle = \{h[i,j]:1\leq i,j<n,\ \text{$i,j$ are of opposite parity}\}
\]
is the set of all white blocks and inverse blocks.  Now put
\[
D_1 = \{ h[i,j]\in D\cap\langle E_n\rangle : \text{$i$ is odd} \}
\quad\text{and}\quad
D_2 = \{ h[i,j]\in D\cap\langle E_n\rangle : \text{$i$ is even} \}.
\]

\begin{lem}\label{lem:J}
The sets $D_1$ and $D_2$ are distinct $\J$-classes of $\langle E_n\rangle$.  Furthermore, $D_1$ and $D_2$ are incomparable in the order on $\J$-classes.
\end{lem}

\begin{proof}
It follows from the defining relation (2) that all elements of $D_1$ are $\D$-related (and hence $\J$-related) to each other, and similarly for $D_2$.  To complete the proof of the first statement, by symmetry, it remains to show that any element $x\in\langle E_n\rangle$ that is $\J$-related to $h[1,2]$ must belong to $D_1$.  So suppose $x$ is such an element.  In particular, $x$ is $\J$-related to $h[1,2]$ in $K_n$, so it follows from above-mentioned facts from~\cite{LFG} that $x=h[i,j]$ for some $i,j$.  But, since $x\in\langle E_n\rangle$, it follows from Theorem \ref{main1} that $i,j$ are of opposite parity.  If $i$ is odd, then $x\in D_1$ and we are done, so suppose instead that $i$ is even.  Since then $h[i,j]\,\J\,h[2,1]$, we deduce that $h[1,2]\,\J\,h[2,1]$, and so $h[1,2]=yh[2,1]z$ for some $y,z\in\langle E_n\rangle$.  It is easy to see, diagrammatically, that $z$ must contain both components $\{2,3\}$ and $\{2',3'\}$.  But then, in fact, $z=h[2,2]$ is a red block and hence not an element of $\langle E_n\rangle$, by Theorem~\ref{main1}, a contradiction.  As noted above, this completes the proof of the first statement.  

We have already seen that $h[1,2]\not=yh[2,1]z$ for all $y,z\in\langle E_n\rangle$, from which it follows that $D_1\not\leq D_2$.  By a symmetrical argument, we also obtain $D_2\not\leq D_1$.
\end{proof}

Note that if $n=2m+1$ is odd, then both $D_1$ and $D_2$ have $m$ $\R$-classes and $m$ $\L$-classes.  On the other hand, if $n=2m$ is even, then $D_1$ has $m$ $\R$-classes and $m-1$ $\L$-classes, with $D_2$ having $m-1$ $\R$-classes and $m$ $\L$-classes.
The $\J$-classes $D_1$ and $D_2$ are pictured in Figure \ref{fig:eggbox} (for $n=10$); in the diagram, $\R$-related elements are in the same row, $\L$-related elements in the same column, and idempotents are shaded grey (such diagrams are commonly called \emph{eggbox diagrams}).

\begin{figure}[ht]
\begin{center}
\begin{tikzpicture}[scale=1]
\foreach \x/\y in {0/4,0/3,1/3,1/2,2/2,2/1,3/1,3/0} {\fillbox{\x}{\y}}
\foreach \x in {0,...,4} {\draw (\x,0)--(\x,5);}
\foreach \x in {0,...,5} {\draw (0,\x)--(4,\x);}
\foreach \x in {1,3,5,7,9} \foreach \y in {2,4,6,8} {\draw(\y/2-.5,{4.5-(\x-1)/2})node{{\tiny $h[\x, \y]$}};}
\draw[|-|] (-.5,0)--(-.5,5);
\draw(-.5,2.5)node[left]{$D_1$};
\begin{scope}[shift={(6,.5)}]
\foreach \x/\y in {0/3,1/3,1/2,2/2,2/1,3/1,3/0,4/0} {\fillbox{\x}{\y}}
\foreach \x in {0,...,5} {\draw (\x,0)--(\x,4);}
\foreach \x in {0,...,4} {\draw (0,\x)--(5,\x);}
\foreach \x in {2,4,6,8} \foreach \y in {1,3,5,7,9} {\draw({(\y-1)/2+.5},{4-(\x-1)/2})node{{\tiny $h[\x, \y]$}};}
\draw[|-|] (5.5,0)--(5.5,4);
\draw(5.5,2)node[right]{$D_2$};
\end{scope}
%
%
%
\end{tikzpicture}
    \caption{Eggbox diagrams of the $\J$-classes $D_1$ and $D_2$ in $\langle E_{10}\rangle$.}
    \label{fig:eggbox}
   \end{center}
 \end{figure}
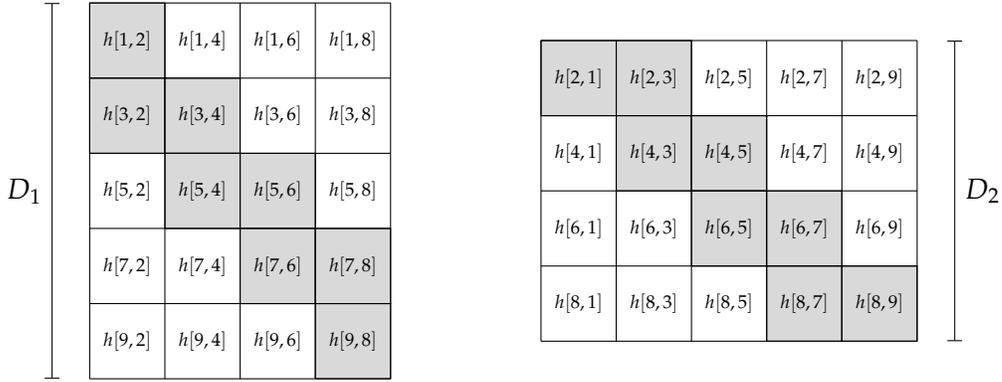


Note that $E_n'\subseteq D_1\cup D_2$.  Since $\langle E_n\rangle=\langle E_n'\rangle$, it follows that $D_1$ and $D_2$ are precisely the maximal $\J$-classes of $\langle E_n\rangle\setminus\{1\}$.  Note also that $E(D_i)$ generates the principal factor $D_i^\star$ (as a semigroup) for each $i$.  (Indeed, if for example ${x\in D_1}$, then $x=e_1\ldots e_k$ for some $e_i\in E_n'$; but if any of the $e_i$ belonged to $D_2$, then we would have $D_1\leq D_2$, contradicting Lemma \ref{lem:J}.)

Since the identity element $1$ cannot be obtained as a (non-vacuous) product of elements of $E_n'$, it follows that the (idempotent) rank of $\langle E_n\rangle$ is equal to the sum of the (idempotent) ranks of the principal factors $D_1^\star$ and $D_2^\star$, where here we consider generation of $D_i^\star$ as semigroups.  

Since each $D_i^\star$ is idempotent generated, \cite[Corollary 8]{Gr} says that $\rank(D_i^\star)$ is equal to the maximum of the number of $\R$- and $\L$-classes contained in $D_i$.  As noted above, this is $m=\lfloor\frac n2\rfloor$, regardless of whether $n=2m$ is even or $n=2m+1$ is odd.  

On the other hand, each $D_i$ contains $n-2$ idempotents, and it turns out that $E(D_i)$ constitutes a unique minimal 
idempotent generating set for the principal factor $D_i^\star$.  
Indeed, by removing an arbitrary element $e$ from $E(D_i)$, one of two things happens (see Figure \ref{fig:eggbox}):
\begin{itemize}
\item[(i)] $E(D_i)\setminus\{e\}$ has empty intersection with an $\R$- or $\L$-class of $D_i$ (for example, if $e=h[1,2]$), or
\item[(ii)] $E(D_i)\setminus\{e\}$ splits into two subsets $X_i,Y_i$ such that no idempotent from $X_i$ is $\L$- or 
$\R$-related to any idempotent from $Y_i$. 
\end{itemize}
In either case, it follows that $\langle E(D_i)\setminus\{e\}\rangle$ does not contain $e$.  Indeed, this follows from \cite[Exercise 12, p98]{Howie} in case (i), or from the proof of \cite[Theorem 1]{Ho} in case (ii).
%
%
Putting all this together, we have proved the following result.

\begin{thm}\label{main2}
Let $n\geq 3$. Then $\mathrm{rank}(\langle E_n\rangle)=2\lfloor\frac{n}{2}\rfloor$ and 
$\mathrm{idrank}(\langle E_n\rangle)=2n-4$. \hfill$\Box$
\end{thm}

\begin{rmk}
The previous result concerns \emph{monoid} generating sets; for the (idempotent) rank in the context of \emph{semigroup} generating sets, $1$ must be added to the above expressions.  Note also that $\rank(\langle E_n\rangle)=\mathrm{idrank}(\langle E_n\rangle)=0$ if $n\leq2$.  By consulting Theorem \ref{main2}, the only other values of $n$ for which $\mathrm{rank}(\langle E_n\rangle)=\mathrm{idrank}(\langle E_n\rangle)$ holds are $n=3,4$.
\end{rmk}

\section*{Acknowledgements}

The first named author gratefully acknowledges the support of Grant No.\ 174019 of the Ministry of Education, Science, and Technological Development of the Republic of Serbia.

\footnotesize
\def\bibspacing{-1.1pt}
\bibliography{GMJ-16-0034}
\bibliographystyle{plain}
\end{document}